\documentclass{amsart}
\usepackage{mathrsfs}
\newtheorem{theorem}{Theorem}[section]
\newtheorem{lemma}[theorem]{Lemma}

\theoremstyle{definition}

\theoremstyle{remark}

\numberwithin{equation}{section}

\begin{document}
	\title{On Maps that Preserve the Lie Products Equal to Fixed Elements}
	\author{Shiv Kumar Chaudhary} 
	\address{Department of Mathematics,	Indian Institute of Technology Patna, Patna-801106}
	\curraddr{}
	\email{E-mail: shiv\_2421ma17@iitp.ac.in} 
	\thanks{}
	
	
	\author{Om Prakash$^{\star}$}
	\address{Department of Mathematics, Indian Institute of Technology Patna, Patna-801106}
	\curraddr{}
	\email{om@iitp.ac.in}
	\thanks{* Corresponding author}

	
	\date{}
	
	\dedicatory{}
	\maketitle
	\begin{abstract}
		   This work characterizes the general form of a bijective linear map $\Psi:\mathscr{M}_n(\mathbb{C}) \to \mathscr{M}_n(\mathbb{C})$ such that $[\Psi(A_1),~\Psi(A_2)]=D_2$ whenever $[A_1,~A_2]=D_1$ where $D_1~\text{and}~D_2$ are fixed matrices. Additionally, let $\mathscr{H}_1$ and $\mathscr{H}_2$ be the infinite-dimensional complex Hilbert spaces. We characterize the bijective linear map $\Psi: \mathscr{B}(\mathscr{H}_1) \to \mathscr{B}(\mathscr{H}_2)$ where $\Psi(A_1) \circ ~\Psi(A_2)=D_2$ whenever $A_1\circ ~A_2=D_1$ and $D_1~\text{and}~D_2$ are fixed operators.  \\
	\end{abstract}
	
	\section{Introduction}
	Let $\mathscr{M}_n$ be the algebra consisting of all $n\times n$ matrices with entries from the complex field $\mathbb{C}$. The Lie product of $A_1, A_2 \in \mathscr{M}_n$ is $[A_1,~A_2]=A_1A_2-A_2A_1$ and the Jordan product is $A_1 \circ A_2=A_1A_2+A_2A_1$.\\\\ In 2021, Julius \cite{jul} raised the question of characterizing bijective linear maps that preserve Lie products equal to some fixed elements as follows.\\

 \textbf{Problem 1:}\label{prob} Let $D_1\in \mathscr{M}_n$ be a fixed matrix, and let $\Psi: \mathscr{M}_n \to \mathscr{M}_n$ be a bijective linear map such that $\Psi(D_1)=D_1$ and
	$$
	[\Psi(A_1),~\Psi(A_2)] = D_1 \quad \text{whenever} \quad  [A_1,~A_2] = D_1.
	$$\\
Can we characterize or describe the structure of the map $\Psi$?\\
This paper answers his problem and obtains the general form of such a bijective linear map.\\

In fact, Problem \ref{prob} has been studied earlier in various partial forms. The most familiar case of the problem \ref{prob} was studied by Watkins \cite{wat} where $D_1=0$, and he proved that if a bijective linear map $\Psi:\mathscr{M}_n(F) \to \mathscr{M}_n(F)$, $n\geq4$ such that $[A_1,~A_2]=0$ implies that $[\Psi(A_1),~\Psi(A_2)]=0$ for all $A_1,A_2 \in \mathscr{M}_n$. Then there exists a linear functional $\eta$ on $\mathscr{M}_n$ and a non-singular matrix $U$ such that $$\Psi(A_1)=cU^{-1}A_1U+\eta(A_1)I$$ or $$\Psi(A_1)=cU^{-1}A_1^tU+\eta(A_1)I$$ for all $A_1 \in \mathscr{M}_n$, $c$ is a scalar, $A_1^t$ denotes the transpose of $A_1$, $I$ is an $n\times n$ identity matrix, and $F$ is an algebraically closed field. Maps of these forms are called commutativity-preserving. Later, it was studied by many algebraists over various algebras \cite{bei1,bei2,bei3,bre1,bre2,Lin}. For more details, we refer to \cite{shi,cat2}.

Recall that, for each pair of indices $i,j \in \{1,2, \dots n \}$, let $e_{ij}$ be the matrix having $1$ in the $(i,j)$-positions and $0$ elsewhere. After a long gap, in 2020, Ginsburg et al. \cite{gin} demonstrated that problem \ref{prob} for a rank-one nilpotent matrix. They showed that if a bijective linear map $\Psi:\mathscr{M}_n \to \mathscr{M}_n$, $n\geq5$ satisfies $[\Psi(A_1),~ \Psi(A_2)]=e_{12}$ with $[A_1,~A_2]=e_{12}$ and $\Psi(e_{12})=e_{12}$ for every $A_1,A_2 \in \mathscr{M}_n$. Then there exists a linear functional $\eta$ on $\mathscr{M}_n$ and a non-singular matrix $U$ such that for $(i,j)\neq(2,1)$, $$\Psi(e_{ij})=U^{-1}e_{ij}U+\eta(e_{ij})I$$ or $$\Psi(e_{ij})=-U^{-1}e_{ji}U+\eta(e_{ij})I.$$ In either case, $\Psi(e_{21}) = ce_{21} + X$ where $c$ is non-zero scalar and $X \in \mathscr{M}_n$ has a zero $(2,1)$-entry. \\
Recently, H. Julius \cite{jul} proved the problem \ref{prob} for rank-two trace-zero matrices. Basically, he proved that if a bijective linear map $\Psi: \mathscr{M}_n \to \mathscr{M}_n$, $n \geq 5$ satisfies $[\Psi(A_1),~\Psi(A_2)]=e_{11}-e_{22}$ with $[A_1,~A_2]=e_{11}-e_{22}$ and $\Psi(e_{11}-e_{22})=e_{11}-e_{22}$ for every $A_1, A_2 \in \mathscr{M}_n$. Then there exists a linear functional $\eta$ on $\mathscr{M}_n$ and a non-singular matrix $U$  such that $$\Psi(A_1)=U^{-1}A_1U+\eta(A_1)I$$ or $$\Psi(A_1)=-U^{-1}A_1^tU+\eta(A_1)I$$ for all $A_1 \in \mathscr{M}_n$.\\

In this paper, we solve the problem \ref{prob} and obtain the most general form, i.e., maps preserving Lie products equal to fixed matrices, without the condition $\Psi(D_1)=D_1$. Of course, we have nice descriptions for rank-one and rank-two trace-zero matrices. We denote $sl_n$ to be the set consisting of all matrices with trace equal to zero and prove the following. \\
\begin{theorem}
    \label{th1}
	Let $D_1,~D_2 \in sl_n$ be fixed matrices and a bijective linear map $\Psi : sl_n \to sl_n$ satisfying $$[\Psi(A_1),~\Psi(A_2)]=D_2 ~~\text{whenever}~~ [A_1,~A_2]=D_1 $$ for all $A_1,A_2 \in sl_n$. Then there exists a non-singular matrix $U$ and a non-zero scalar $c$ such that $$\Psi(A_1)=cUA_1U^{-1}$$ or $$\Psi(A_1)=cUA_1^tU^{-1}$$ for all $A_1 \in sl_n$.
\end{theorem}

The above result represents the standard form of a map that preserves commutativity on $sl_n$ (for details see Lin \cite{Lin}). Furthermore, an analogous characterization can be derived for such a map on the complete matrix algebra $\mathscr{M}_n$. Towards this, we have the following.

\begin{theorem}
	\label{th2}
	Let $D_1,D_2 \in \mathscr{M}_n$ be the fixed matrices and a bijective linear map $\Psi : \mathscr{M}_n \to \mathscr{M}_n$ satisfying $$[\Psi(A_1),~\Psi(A_2)]=D_2 ~~\text{whenever}~~ [A_1,~A_2]=D_1 $$ for all $A_1,A_2 \in \mathscr{M}_n$. Then there exists a non-singular matrix $U$, a non-zero scalar $c$ and a linear functional $\eta : \mathscr{M}_n \to \mathbb{C}$ such that $$\Psi(A_1)=cUA_1U^{-1}+\eta(A_1)I$$ or $$\Psi(A_1)=cUA_1^tU^{-1}+\eta(A_1)I$$ for all $A_1 \in \mathscr{M}_n$.
\end{theorem}

If we replace the Lie product with the Jordan product in Theorem \ref{th2}, we obtain the result that was previously proved by L. Catalano et al. \cite{cat1}. For the most trivial case, see \cite{che}. We generalize those results here for infinite dimensions. Let $\mathscr{H}_1$ and $\mathscr{H}_2$ be the infinite-dimensional complex Hilbert spaces, and let $\mathscr{B}(\mathscr{H}_1)$ and $ \mathscr{B}(\mathscr{H}_2)$ be the corresponding algebras of bounded linear operators. We say an additive map $\Psi: \mathscr{B}(\mathscr{H}_1) \to \mathscr{B}(\mathscr{H}_2)$ is square zero preserving if $A_1^2=0$, then $\Psi(A_1)^2=0$ for every $A_1 \in \mathscr{B}(\mathscr{H}_1)$ and an idempotent preserving in both directions if $A_1^2=A_1$ if and only if $\Psi(A_1)^2=\Psi(A_1)$ for every $A_1 \in \mathscr{B}(\mathscr{H}_1)$. Basically, we will prove the following.

\begin{theorem}
    \label{th3}
	Let $D_1 \in \mathscr{B}(\mathscr{H}_1)$ and $D_2 \in \mathscr{B}(\mathscr{H}_2)$ are fixed operators, and let $\Psi : \mathscr{B}(\mathscr{H}_1) \to \mathscr{B}(\mathscr{H}_2)$ be a linear map such that $$\Psi(A_1) \circ \Psi(A_2)=D_2 ~~\text{whenever}~~ A_1 \circ A_2=D_1$$ for every $A_1,A_2 \in \mathscr{B}(\mathscr{H}_1)$. Then $\Psi$ preserves square-zero operators.
\end{theorem}

\begin{theorem}
\label{th4}
    Let $D_1 \in \mathscr{B}(\mathscr{H}_1)$ and $D_2 \in \mathscr{B}(\mathscr{H}_2)$ are fixed operators, and let $\Psi : \mathscr{B}(\mathscr{H}_1) \to \mathscr{B}(\mathscr{H}_2)$ be a bijective linear map such that $$\Psi(A_1) \circ \Psi(A_2)=D_2 ~~\text{whenever}~~ A_1 \circ A_2=D_1$$ for every $A_1,A_2 \in \mathscr{B}(\mathscr{H}_1)$. Then there exists a non-zero scalar $c$ and an invertible operator $U \in \mathscr{B}(\mathscr{H}_1, \mathscr{H}_2)$ such that $$\Psi(A_1)=cUA_1U^{-1}$$ or $$\Psi(A_1)=cUA_1^\star U^{-1},$$ where $A_1^\star$ denotes the transpose of $A_1$ related to a fixed but arbitrary orthonormal base of $\mathscr{H}_1$.
\end{theorem}

Before going to the main proof, we need a few lemmas. Here, the set $\mathcal{C}(A_1) = \{P\in \mathscr{M}_n~\mid~ A_1P=PA_1\}$ represents the centralizer-subalgebra of $A_1 \in \mathscr{M}_n$.
\begin{lemma}
	\label{lem1}
    Let $\Psi : sl_n \to sl_n$ be a bijective linear map such that
    \begin{equation}
    \label{eq1}
        [\Psi(A_1),~\Psi(A_2)]=D_2 ~~\text{whenever}~~ [A_1,~A_2]=D_1.
    \end{equation}
    Then
    \begin{enumerate}
      \item  $\Psi(\mathcal{C}(A_1)) \subseteq \mathcal{C}(\Psi(A_1))$ and $\Psi(\mathcal{C}(A_2)) \subseteq \mathcal{C}(\Psi(A_2))$.
      \item  $[\Psi(S),~\Psi(T)]=0$ whenever $[S,~T]=0$ with $S \in \mathcal{C}(A_2), T \in \mathcal{C}(A_1)$.
      \item  $\Psi(A')=B'$, where $A'$ and $B'$ are rank-one trace-zero matrix.
    \end{enumerate}
\end{lemma}

\begin{proof} $1.$
     Let $M_1 \in \mathcal{C}(A_1)$. Then $[A_1,~A_2+M_1]=D_1$ and using \ref{eq1}, we get
     \begin{align*}
         D_2=&[\Psi(A_1),~\Psi(A_2+M_1)]\\
          =&[\Psi(A_1),~\Psi(A_2)]+[\Psi(A_1),~\Psi(M_1)].
     \end{align*}
     Again, from \ref{eq1}, we have $[\Psi(A_1),~\Psi(A_2)]=D_2$. Therefore, $$[\Psi(A_1),~\Psi(M_1)]=0.$$
     Similarly, let $N_1 \in \mathcal{C}(A_2)$. Then $[A_1+N_1,~A_2]=D_1$ and using \ref{eq1}, we get
     \begin{align*}
         D_2=&[\Psi(A_1+N_1),~\Psi(A_2)]\\
          =&[\Psi(A_1),~\Psi(A_2)]+[\Psi(N_1),~\Psi(A_2)].
     \end{align*}
     Again, using \ref{eq1}, we have $$[\Psi(N_1),~\Psi(A_2)]=0.$$\\
	$2.$ If $S \in \mathcal{C}(A_2)$, $T \in \mathcal{C}(A_1)$ and $[S,~T]=0$, then $[A_1+S,~A_2+T]=D_1$. By using \ref{eq1}, we get
    \begin{align*}
         D_2=&[\Psi(A_1+S),~\Psi(A_2+T)]\\
          =&[\Psi(A_1),~\Psi(A_2)]+[\Psi(A_1),~\Psi(T)]+[\Psi(S),~\Psi(A_2)]+[\Psi(S),~\Psi(T)].
     \end{align*}
     Align with part $1$, we have $[\Psi(A_1),~\Psi(T)]=0,~[\Psi(S),~\Psi(A_2)]=0$. Therefore, $[\Psi(S),~\Psi(T)]=0.$\\
     $3.$ 
     Let $A_1 \in sl_n$, $I$ is identity matrix of order $n$ and a scalar $c \in \mathbb{C}$, now we define a map $\eta: \mathscr{M}_n \to \mathscr{M}_n$ by $$\eta(A_1+cI)=\Psi(A_1)+cI.$$ One can easily check that the map $\eta$ is bijective and linear, and agrees with $\Psi$ on every trace-zero matrix. By part $1$, if $A' \in sl_n$ is a rank-one matrix then $$\text{dim}~\mathcal{C}(\eta(A')) \geq n^2-2n+2.$$ If $\text{dim}~\mathcal{C}(\eta(A')) > n^2-2n+2$, then by [\cite{wat},~Lemma] there exists $n \times n$ identiy matrix $I$ and a scalar $\lambda \in \mathbb{C}$ such that $$\eta(A')=\lambda I.$$ But then $\lambda$ must be $0$, since $\eta(A')=\Psi(A') \in sl_n$, contradicting the bijectivity. Hence $$\text{dim}~\mathcal{C}(\eta(A')) = n^2-2n+2.$$ Again, by [\cite{wat}, ~Lemma], there exists a scalar $\lambda'$ and a rank-one matrix $B'$ such that $$\eta(A')=B'+\lambda'I.$$ If trace of $B'$ is zero, then by assumption, $B'+\lambda'I$ is trace-zero matrix, and implies that $\lambda'=0$. Hence, $$\Psi(A')=B',$$ where $B'$ is rank-one trace-zero matrix.\\
     Otherwise, we know that matrix $B'$ can be written uniquely as $B''+\lambda''I$ where $B''$ is a rank-one trace-zero matrix and $\lambda''$  is some scalar. Therefore, we have $$\eta(A')=B''+\lambda'''I$$ and by under same assumption, we have $\Psi(A')=B''$.
\end{proof}

\begin{lemma}
	\label{lem2}
	Suppose $\Psi: \mathscr{B}(\mathscr{H}_1) \to \mathscr{B}(\mathscr{H}_2)$ is a linear map satisfies the conditions of Theorem \ref{th3} if and only if it satisfies
    \begin{equation}
        \label{eq2}
        2 \Psi(X_1)^2-2\Psi(Y_1)^2=D_2~~\text{whenever}~~2X_1^2-2Y_1^2=D_1
    \end{equation}
    for all $X_1, Y_1 \in \mathscr{B}(\mathscr{H}_1)$.
\end{lemma}
\begin{proof}
	Putting $A_1=X_1+Y_1$ and $A_2=X_1-Y_1$ in the forward direction, we get
	$$\Psi(X_1+Y_1) \circ \Psi(X_1-Y_1)=D_2 ~~\text{whenever}~~ (X_1+Y_1) \circ (X_1-Y_1)=D_1,$$ and using Linearity of $\Psi$ and appropriate calculations, we get $$ 2 \Psi(X_1)^2-2\Psi(Y_1)^2=D_2 ~~\text{whenever}~~2X_1^2-2Y_1^2=D_1$$ for every $X_1, Y_1 \in \mathscr{B}(\mathscr{H}_1)$.\\
    Similarly, one can prove by substituting $X_1=\frac{A_1+A_2}{2}$ and $Y_1=\frac{A_1-A_2}{2}$ in the backward direction.
\end{proof}

\begin{lemma}
	\label{lem3}
	Let a linear map $\Psi: \mathscr{B}(\mathscr{H}_1) \to \mathscr{B}(\mathscr{H}_2)$ satisfying \ref{eq2}. Then it follows that
    \begin{equation}
        \label{eq3}
        \Psi(P_1)^2=\Psi(Q_1)^2 ~~\text{whenever}~~P_1^2=Q_1^2=cI
    \end{equation}
    for every $P_1, Q_1 \in \mathscr{B}(\mathscr{H}_1)$ and a fixed scalar $c$.
\end{lemma}
\begin{proof}
	Let $\Psi$ satisfy \ref{eq2}, and let $P_1, Q_1 \in \mathscr{B}(\mathscr{H}_1)$ be such that $P_1^2=Q_1^2=cI$. Let $cI- \frac{D_1}{2} \in \mathscr{B}(\mathscr{H}_1)$ be a positive operator (and hence, has a square root), we know there exists $T_1 \in \mathscr{B}(\mathscr{H}_1)$ such that $$2P_1^2-2T_1^2=2Q_1^2-2T_1^2=D_1.$$ We have that $$2\Psi(P_1)^2-2\Psi(T_1)^2=2\Psi(Q_1)^2-2\Psi(T_1)^2=D_2.$$ Hence, $$\Psi(P_1)^2=\Psi(Q_1)^2.$$
\end{proof}
\begin{lemma}
    \label{lem4}
    Let a linear map $\Psi: \mathscr{B}(\mathscr{H}_1) \to \mathscr{B}(\mathscr{H}_2)$ satisfying \ref{eq3}. Then it is a square-zero preserving map.
\end{lemma}
\begin{proof}
    Let $N_1 \in \mathscr{B}(\mathscr{H}_1)$ be a square-zero operator (i.e., $N_1^2=0$) and $M_1 \in \mathscr{B}(\mathscr{H}_1)$ be an operator such that $N_1 \circ M_1=0$ and $M_1^2=cI$. Then $$(N_1+M_1)^2=(N_1-M_1)^2=cI.$$ From \ref{eq3}, we have $$\Psi(N_1+M_1)^2=\Psi(N_1-M_1)^2$$ and after using linearity of $\Psi$, we get $$\Psi(N_1) \circ \Psi(M_1)=0.$$ We can see that $$M_1^2=(N_1+M_1)^2=cI,$$ and again using \ref{eq3} yields $$\Psi(M_1)^2=\Psi(N_1+M_1)^2.$$ Since linearity of $\Psi$ gives us $$\Psi(M_1)^2=\Psi(N_1)^2+\Psi(M_1)^2+\Psi(N_1)\circ \Psi(M_1)$$ and using $\Psi(N_1)\circ \Psi(M_1)=0$, finally we get $$\Psi(N_1)^2=0.$$ Hence, $\Psi$ preserves square-zero operators.
\end{proof}

Proof of Theorem \ref{th1}.
\begin{proof}
	Let $A' \in sl_n$ be a rank-one matrix. We aim to show that $\Psi$ preserves $A'$ to a rank-one trace-zero matrix. By Lemma \ref{lem1}, we have $$\Psi(A')=B'$$ where $B'$ is a rank-one trace-zero matrix. By Marcus and Moyls [\cite{Mar}, Theorem 1], if $\phi: \mathscr{M}_n \to \mathscr{M}_n$ is a linear map such that $\text{rank}(\phi(X))=1$ whenever $\text{rank}(X)=1$ for every $X \in \mathscr{M}_n$, then $$\phi(A_1)=UA_1V$$ or $$\phi(A_1)=UA_1^tV$$ for all $A_1 \in \mathscr{M}_n$, where $U$ and $V$ are non-singular matrices.

    Let $A_1 \in sl_n$, $I$ is identity matrix of order $n$ and a scalar $c \in \mathbb{C}$, now we define a map $\phi: \mathscr{M}_n \to \mathscr{M}_n$ by $$\phi(A_1+cI)=\Psi(A_1)+cI.$$ It is easy to check that the map $\phi$ is bijective and linear, and agrees with $\Psi$ on every trace-zero matrix. 

    Now, since $\phi(A_1)=\Psi(A_1)$ for every $A_1 \in sl_n$, this implies $$\text{tr}(\phi(A_1))=0$$ for all $A_1 \in sl_n$. If $\phi(A_1)=UA_1V$, then $$0=\text{tr}(\phi(A_1))=\text{tr}(UA_1V)=\text{tr}(VUA_1)$$ for all $A_1\in sl_n$. Therefore, it follows that $VU=cI$ for some non-zero scalar $c$. Hence, $$\phi(A_1)=cUA_1U^{-1}$$ for all $A_1 \in sl_n$. Similarly, if $\phi(A_1)=UA_1^tV$, then $$\phi(A_1)=cUA_1^tU^{-1}$$ for all $A_1 \in sl_n$.
\end{proof}
Proof of Theorem \ref{th2}.
\begin{proof}
	Now, we will extend $\Psi$ to complete matrix algebra $\mathscr{M}_n$. As every matrix $A_1 \in \mathscr{M}_n$ can be uniquely written as $$A_1 = A_1' + \lambda I$$ for some scalar $\lambda$ and a trace-zero matrix $A_1'$. 
    It suffices to show that $\Psi(I)=cI$. By parts $1$ and $2$ of Lemma \ref{lem1}, we have that $$[\Psi(I),~\Psi(Y)]=0$$ for all $Y \in \mathcal{C}(A_1)$, and $$[\Psi(X),~\Psi(I)]=0$$ for all $X \in \mathcal{C}(A_2)$. Thus, $$\text{dim}~\mathcal{C}(\Psi(I)) > n^2-2n+2.$$ Using [\cite{wat}, Lemma], we get $$\Psi(I)=cI$$ for some scalar $c$.
    Now, if $A_1 \in sl_n$, then $\Psi(A_1)=B_1+k I$ for some $B_1 \in sl_n$ and $k \in \mathbb{C}$, so there exists a linear functional $\eta$ such that the map $\Psi(A_1)-\eta(A_1)I$ is a bijective and linear, and the resulting matrix is trace-zero whenever $A_1$ is trace-zero. We can define $\psi: sl_n \to sl_n$ such that $$\psi(A_1)=\Psi(A_1)-\eta(A_1)I$$ and apply Theorem \ref{th1}.
\end{proof}
Proof of Theorem \ref{th3}.
\begin{proof}
    We can see that Lemmas \ref{lem2}, \ref{lem3}, and \ref{lem4} together imply that $\Psi$ preserves square-zero operators.
\end{proof}
Proof of Theorem \ref{th4}.
\begin{proof}
By Lemma \ref{lem4}, any map $\Psi: \mathscr{B}(\mathscr{H}_1) \to \mathscr{B}(\mathscr{H}_2)$ satisfies the assumptions in Theorem \ref{th4} preserves square-zero operators, and  [\cite{bai}, Corollary 2.2] gives us if any bijective linear map $\Psi: \mathscr{B}(\mathscr{H}_1) \to \mathscr{B}(\mathscr{H}_2)$ preserves square-zero operators. Then
\begin{align*}
    &\Psi(L) \Psi(I)+ \Psi(I)\Psi(L)=2\Psi(L)^2,\\
    & \Psi(I)^2 \Psi(L)=\Psi(L) \Psi(I)^2,
\end{align*}
and $ \Psi(L)^2\Psi(I)=\Psi(I)\Psi(L)^2$ for all idempotent $L \in \mathscr{B}(\mathscr{H}_1)$.\\
It follows from \cite{pea} that every element in $\mathscr{B}(\mathscr{H}_1)$ can be written as a sum of at most five idempotents. It means, $\Psi(I)^2 \Psi(X)=\Psi(X) \Psi(I)^2$ holds for each $X \in \mathscr{B}(\mathscr{H}_1)$. The bijectivity of $\Psi$ gives us $$\Psi(I)^2=cI,$$ for some non-zero scalar $c$. Without loss of generality, we assume $\Psi(I)^2=I$. Then there exists an idempotent $L_1$ such that $$\Psi(I)=2L_1 -I.$$ Using the same argument as above to $\Psi ^{-1}$, we get that for some non-zero scalar $\lambda$ such that $\Psi ^{-1}(I)^2=\lambda I$, and for each idempotent $L_2 \in \mathscr{B}(\mathscr{H}_2)$ such that
$$\Psi ^{-1}(L_2) \Psi ^{-1}(I)+ \Psi ^{-1}(I)\Psi ^{-1}(L_2)=2\Psi ^{-1}(L_2)^2 $$ and $$ \Psi ^{-1}(I) \Psi ^{-1}(L_2)^2=\Psi ^{-1}(L_2)^2 \Psi ^{-1}(I)$$ holds. Since $\Psi(I)=2L_1 -I$, one gets $\Psi ^{-1}(I)=2 \Psi^{-1}(L_1)-I$, and using it we get
 $$\Psi ^{-1}(L_2) (2\Psi ^{-1}(L_1)-I)+ (2\Psi ^{-1}(L_1)-I)\Psi ^{-1}(L_2)=2\Psi ^{-1}(L_2)^2 $$ holds for all idempotents $L_2$. Let $L_2=L_1$ then $$\Psi ^{-1}(L_1) (2\Psi ^{-1}(L_1)-I)+ (2\Psi ^{-1}(L_1)-I)\Psi ^{-1}(L_1)=2\Psi ^{-1}(L_1)^2 $$ after some calculations we get $$\Psi^{-1}(L_1)=\Psi^{-1}(L_1)^2.$$ Since $L_3=\Psi^{-1}(L_1)$ is an idempotent. It clear that $$\Psi^{-1}(I)^2=I.$$ Also, $\Psi(L)\Psi(I)+\Psi(I)\Psi(L)=2\Psi(L)^2$ and using $\Psi(I)=2L_1-I$, we get $$\Psi(L)(2L_1-I)+(2L_1-I)\Psi(L)=2\Psi(L)^2.$$ Simplifying above, we have $L_1\Psi(L)+\Psi(L)L_1=\Psi(L)^2+\Psi(L).$ It gives us
\begin{align*}
    (L_1-\Psi(L))^2=& L_1^2+\Psi(L)^2-L_1\Psi(L)-\Psi(L)L_1\\
                      =& L_1+\Psi(L)^2-\Psi(L)^2-\Psi(L)\\
                      =&L_1-\Psi(L)
\end{align*}
and hence $\Psi(L_2-L)^2=\Psi(L_2-L)$.
Let $N_1 \in L_2\mathscr{B}(\mathscr{H}_1)(1-L_2)$ be an element. Then $ N_1^2=0 ~~\text{and}~~(L_2+N_1)^2=L_2+N_1.$ Since $\Psi$ preserves square zero operators,
\begin{align*}
    0=&\Psi(N_1)^2\\
    =&\Psi(L_2-(L_2+N_1))^2\\
    =&\Psi(L_2-(L_2+N_1))\\
    =&-\Psi(N_1).
\end{align*}
Hence, $N_1=0.$\\
This shows that $L_2\mathscr{B}(\mathscr{H}_1)(1-L_2)=\{0\}$. Therefore, either $L_2=0$ or $L_2=1$. Hence, it is clear that $\Psi^{-1}(I)=\pm I$ and $\Psi(I)=\pm I$. Thus, we have $\Psi(I)=cI$ for some non-zero scalar $c$. Further, without loss of generality, we may assume that $\Psi(I)=I$ and again using $\Psi(L)\Psi(I)+\Psi(I)\Psi(L)=2\Psi(L)^2$, we have $\Psi(L)=\Psi(L)^2$. Thus, we get our desired result using [\cite{bre3}, Theorem 1].
\end{proof}	

\section*{Acknowledgment}
The first author thanks the University Grants Commission (UGC), Govt. of India, for financial support under UGC Ref. No. 231620125825, all authors thank the Indian Institute of Technology Patna for providing the research facilities.

\section*{Declarations}
~~~~~~\textbf{Data Availability Statement}: The authors declare that [the/all other] data supporting the findings of this study are available in the article.\\

\textbf{Disclosure statement}: The authors report that no competing interests exist to declare.\\

\textbf{Use of AI tools Declaration}: The authors declare that they have not used Artificial Intelligence (AI) tools in creating this manuscript.

\end{document}